\documentclass[reqno,onecolumn,oneside]{paper}
\usepackage[english]{babel}

\usepackage{enumerate,amsmath,amsthm,amsfonts,pifont,slashed,amssymb,ifsym,mathrsfs,graphicx,graphics,calligra,
latexsym,txfonts}

\usepackage[colorlinks]{hyperref}
\usepackage[all]{xy}
\usepackage[sort]{cite}

\newtheorem{theorem}{Theorem}[section]
\newtheorem{proposition}[theorem]{Proposition}
\newtheorem{lemma}[theorem]{Lemma}
\newtheorem{corollary}[theorem]{Corollary}
\newtheorem{claim}{Claim}
\theoremstyle{definition}
\newtheorem{definition}[theorem]{Definition}
\newtheorem{exm}[theorem]{Example}

\theoremstyle{remark}
\newtheorem{remark}[theorem]{Remark}

\newcommand{\Lex}{\operatorname{Lex}}
\newcommand{\Min}{\operatorname{Min}}

\newcommand{\ob}{\operatorname{Obj}}

\newcommand{\op}{\operatorname{op}}

\newcommand{\obj}{\operatorname{Obj}}

\newcommand{\dom}{\operatorname{dom}}

\newcommand{\id}{\operatorname{id}}

\newcommand{\Spec}{\operatorname{Spec}}
\newcommand{\Id}{\operatorname{Id}}
\renewcommand{\Im}{\operatorname{Im}}
\newcommand{\Rad}{\operatorname{Rad}}

\newcommand{\alg}{\ell\cat{G}^\text{Ab}}

\newcommand{\supp}{\operatorname{supp}}

\newcommand{\lex}{\times_{\operatorname{lex}}}

\newcommand{\BB}{\mathcal{B}\!\operatorname{oole}}
\renewcommand{\O}{\mathbb{O}}
\newcommand{\E}{\mathbb{E}}
\newcommand{\Set}{\mathcal{S}\!\operatorname{et}}
\newcommand{\fwp}{\mathscr{F}}

\newcommand{\MV}{\mathcal{MV}}
\newcommand{\MVl}{\mathcal{MV}^{\operatorname{lr}}}

\newcommand{\cat}{\mathcal}

\newcommand{\B}{\operatorname{B}}
\newcommand{\sk}{\operatorname{Sk}}

\newcommand{\Max}{\operatorname{Max}}

\newcommand{\R}{\mathbb{R}}
\newcommand{\N}{\mathbb{N}}
\newcommand{\C}{\mathcal{C}}
\newcommand{\Z}{\mathbb{Z}}

\renewcommand{\wp}{\mathscr{P}}

\newcommand{\TMV}{{}^{\operatorname{MV}}\!\mathcal{T}\!\!\operatorname{op}}
\renewcommand{\Top}{\mathcal{T}\!\!\operatorname{op}}
\newcommand{\Fuz}{\mathcal{F}\!\!\operatorname{uz}}

\newcommand{\LFuz}{\mathcal{LF}\!\!\operatorname{uz}}

\newcommand{\LTMV}{{}^{\operatorname{LMV}}\!\mathcal{T}\!\!\operatorname{op}}

\renewcommand{\phi}{\varphi}

\newcommand{\cou}{{^\leftarrow}}
\newcommand{\fcou}{^{{\rotatebox[origin=c]{180}{$\rightsquigarrow$}}}}
\newcommand{\restr}{\upharpoonright}
\newcommand{\la}{\langle}
\newcommand{\ra}{\rangle}

\newcommand{\lto}{\longrightarrow}
\newcommand{\To}{\Longrightarrow}
\newcommand{\lmapsto}{\longmapsto}

\newcommand{\0}{\mathbf{0}}
\renewcommand{\1}{\mathbf{1}}

\def\amslatex\slash{{\protect\AmS-\protect\LaTeX}}

\begin{document}

\title{MV-algebras as Sheaves of $\ell$-Groups on Fuzzy Topological Spaces}

\author{\renewcommand{\thefootnote}{\arabic{footnote}}
\rm Luz Victoria De La Pava\footnotemark[1] \and
\renewcommand{\thefootnote}{\arabic{footnote}}
\rm Ciro Russo\footnotemark[2]}

\maketitle
\footnotetext[1]{Departamento de Matem\'aticas, Universidad del Valle -- Cali, Colombia. \\ {\tt victoria.delapava@correounivalle.edu.co}}
\footnotetext[2]{Departamento de Matem\'atica, Universidade Federal da Bahia -- Salvador, Bahia, Brazil. \\ {\tt ciro.russo@ufba.br}}
\today

\begin{abstract}
We introduce the concept of fuzzy sheaf as a natural generalisation of a sheaf over a topological space in the context of fuzzy topologies. Then we prove a representation for a class of MV-algebras, that we called ``locally retractive'', in which the representing object is an MV-sheaf of lattice-ordered Abelian groups, namely, a fuzzy sheaf in which the base (fuzzy) topological space is an MV-topological space and the stalks are Abelian $\ell$-groups. Last, we show that any MV-algebra is embeddable in a locally retractive algebra and, therefore, in the algebra of global sections of one of such sheaves.
\end{abstract}

\section{Introduction}
\label{intro}

MV-topological spaces are fuzzy topological spaces in which \L ukasiewicz t-norm and t-conorm paly the role of strong intersection and union of fuzzy sets. They were introduced by the second author \cite{rus2} with the aim of extending Stone duality to semisimple MV-algebras. Many basic notions and results of general topology have been succefully extended to MV-topologies in \cite{rus2} and \cite{dlpr}, and the results obtained so far indicate that MV-topological spaces constitute a pretty well-behaved fuzzy generalization of classical topological spaces. %On the algebraic side, the class of algebras which play for MV-topologies the role that Boolean algebras play for classical topologies is the one of \emph{limit cut complete} (\emph{lcc}, for short) MV-algebras. It is worth noticing that lcc MV-algebras form a reflective subcategory of the category of MV-algebras and a completion of the one of semisimple MV-algebras.

In this paper, we extend the concept of sheaf to fuzzy topological spaces with particular emphasis to the class of MV-topological spaces; then we represent a class of MV-algebras as \emph{MV-sheaves} of lattice-ordered Abelian groups. More precisely, we show that every \emph{locally retractive} MV-algebra, namely, every MV-algebra whose quotients on certain ideals have retractive radicals, is isomorphic to the algebra of global sections of an MV-sheaf of $\ell$-groups. Our representation is strongly connected to Filipoiu and Georgescu's sheaf representation for MV-algebras \cite{G-F}. Indeed, from a strictly algebraic viewpoint, we use essentially the same tool, that is, the fact that any MV-algebra $A$ is subdirectly embeddable in the product of a family of local MV-algebras.

However, our representation differs from the one in \cite{G-F} in the way the ``information is encoded''. In Filipoiu and Georgescu's representation, each MV-algebra is obtained as an algebra of global sections of a (classical) sheaf over the maximal spectrum of the algebra and whose stalks are local MV-algebras. So, grossly speaking, we can say that each element of the algebra is represented as an open set of maximal ideals, carrying just the Boolean information, with an element of a local MV-algebra attached to each of its points, the latter encoding the ``non-idempotent'' part. In our representation, the base space is the maximal MV-spectrum (see \cite{rus2}) and is in charge of encoding the whole semisimple skeleton of the given algebra, while the stalks only carry the non-semisimple (or infinitesimal) information of the elements of the algebra. Therefore, using the same description, each element of the algebra is a fuzzy open set along with $\ell$-group elements attached to its (fuzzy) points; the fuzzy points of the open set form the semisimple part and the group elements represent exclusively the infinitesimal one.

As already stated, such an MV-sheaf representation is given for locally retractive MV-algebras, which seems to be a pretty strong limit, but it is not really so, since we also prove that any MV-algebra can be embedded in a locally retractive algebra (Corollary \ref{retremb}), and therefore in the algebra of global sections of an MV-sheaf of Abelian $\ell$-groups (Corollary \ref{sheafemb}). In this respect, Theorem \ref{mvrthm} plays a key role, together with \cite[Theorem 4.5]{dinespger}.

Throughout the paper, unless otherwise specified, we refer the reader to \cite{mvbook} for any definition about MV-algebras not explicitly reported here.

\section{MV-algebra ideals and lexicographic MV-algebras}
\label{mvalg}

In this preliminary section we shall recall some notions and results on MV-algebras, mainly from \cite{ferraleti} and \cite{dfl}. For all the very basic facts on MV-algebras, we refer the reader to \cite{mvbook}

Given an MV-algebra $\la A, \oplus, ^*, 0\ra$, an \emph{ideal} of $A$ is a downward closed submonoid of $\la A, \oplus, 0\ra$. It is well-known that congruences and ideals of MV-algebras are in bijective correspondence, namely, that each ideal of an algebra $A$ is the class of $0$ for exactly one congruence of $A$. \emph{Maximal ideals} are ideals which are maximal w.r.t. the inclusion relation in the set of all ideals of $A$, while a \emph{prime ideal} is any ideal $I$ such that, for all $a, b \in A$, if $a \wedge b \in I$, then at least one of the two elements of the algebra is in $I$. $\Spec A$ and $\Max A$ denote the sets of, respectively, prime and maximal ideals of $A$. Note that $\Max A \subseteq \Spec A$, and the equality does not hold in general. The \emph{radical} of $A$ is the intersection of all maximal ideals; it is denoted by $\Rad A$. We also denote by $\Min A$ the set of minimal prime ideals of $A$.

An MV-algebra is called \emph{local} if it has a unique maximal ideal (which, consequently, is also the radical). All totally ordered MV-algebras (MV-chains) are local. All the quotients of an MV-algebra over a prime ideal are chains, and therefore local. An ideal $I$ of $A$ is called \emph{primary} if $A/I$ is a local algebra, hence, all prime ideals are primary, while the converse is not true in general.

Let $A$ be an MV-algebra and $P$ a prime ideal of $A$. The set
\begin{equation}\label{O_P}
O_P=\bigcap\{Q\in \Min A \mid Q\subseteq P\}.
\end{equation}
is obviously an ideal of $A$. It is immediate to verify that $O_P = \bigcap\{Q\in \Spec A \mid Q\subseteq P\}$, for all $P \in \Spec A$.

\begin{proposition}\cite{ferraleti}\label{O_P ch}
For each $P\in \Spec A$, $O_P=\bigcup\{a^\perp:a\notin P\}$, where $a^\perp=\{b\in A:a\wedge b=0\}.$
\end{proposition}

\begin{proposition}\cite{ferraleti}\label{O_P is primary}
  For each $P\in\Spec A$, the ideal $O_P$ is primary.
\end{proposition}

\begin{remark}\label{subemb}
Every prime ideal of an MV-algebra $A$ is contained in a maximal ideal, hence $\bigcap_{M \in \Max A} O_M = \bigcap \Spec A = \{0\}$. Then, by Universal Algebra, there exists a subdirect embedding of $A$ into $\prod_{M \in \Max A} A/O_M$. Such an embedding will be the main algebraic tool of our representation.
\end{remark}

We recall that a \emph{partially-ordered Abelian group} is an Abelian group $\left( G,+, -, 0\right)$ endowed with a partial order relation $\leq$ which is compatible with the sum. The positive cone $G_+$ of $G$ is the set $\{x \in G \mid 0 \leq x\}$, while the negative cone $G_-$ is $(-G_+)$, i.e., the set of all the elements of $G$ which are $\leq 0$. When the order relation is total, $G$ is called a \emph{totally-ordered Abelian group} ($o$-group for short), and if the order of $G$ is a lattice order the group is called a \emph{lattice-ordered Abelian group} ($\ell$-group henceforth). An element $u \in G$ is a \emph{strong (order) unit} if $u \geq 0$ and, for every $x \in G$ there is a natural number $n$ such that $x \leq nu$. An $\ell$-group (respectively: an $o$-group) $G$ with a strong unit $u$ is called a \emph{unital $\ell$-group} (\emph{$\ell u$-group}) (resp.: \emph{unital $o$-group}, \emph{$ou$-group}) and is usually denoted by $\left( G, u\right)$.

\begin{definition}\label{retrlexdef}
An ideal $I$ of an MV-algebra $A$ is called \emph{retractive} if the natural projection $A \to A/I$ is a retraction. $I$ is called \emph{lexicographic} if the following hold:
\begin{enumerate}[(LMV1)]
  \item $I\neq\{0\}$,
  \item $I$ is strict, i.e., $\forall a,b \in A (a/I < b/I \To a < b)$,
  \item $I$ is retractive,
  \item $I$ is prime,
  \item $\rho\leq x\leq\rho^*$, for any $\rho\in I$ and any $x\in A\setminus\langle I\rangle$.
\end{enumerate}
The set of all lexicographic ideals of $A$ is denoted by $\Lex\Id(A)$.
\end{definition}

\begin{definition}
An MV-algebra $A$ is called \emph{lexicographic} if $\Lex\Id(A)\neq\emptyset.$
\end{definition}
\begin{theorem}{\cite[Theorem 4.1]{dfl}}\label{reprthlex}
The following are equivalent:
\begin{enumerate}[(a)]
\item $A$ is a lexicographic MV-algebra,
\item there exists an ou-group $(H,u)$ and a non-trivial $\ell$-group $G$ such that $$A\cong\Gamma(H\lex G,(u,0)).$$
\end{enumerate}
\end{theorem}
This representation theorem says that the class of lexicographic MV-algebras is the largest class of MV-algebras which can be represented, via Mundici's functor $\Gamma$ (see \cite[Section 2.1]{mvbook} or \cite{mun}), as lexicographic products of \emph{ou}-groups and non-trivial $\ell$-groups, with strong unit of the form $(u, 0)$. We recall here a sketch of the proof in order to provide the reader with some technical tools that will be used later.
\begin{proof}{(Sketch)}
\begin{itemize}
\item[$\Rightarrow)$] It follows from \cite[Proposition 3.1]{dfl}.
\item [$\Leftarrow)$] Let $A$ be a lexicographic MV-algebra and $I$ a lexicographic ideal of $A$. We have the following:
\begin{itemize}
\item Let $\delta_I$ be the retraction of the canonical projection $\pi_I:A\lto A/I$.
\item Let $S_I=\delta_I(A/I)$ the MV-subalgebra of $A$ which is isomorphic to $A/I$.
\item For any $a\in A$, set $s_a=\delta_I(\pi_I(a))$ as the unique element of $S_I$ such that $[s_a]_I = [a]_I$.
\item Set $\varepsilon_a=a\odot s_a^*$ and $\tau_a=a^*\odot s_a$.
\item There exist an isomorphism of MV-algebras $\zeta_I: S_I \lto \Gamma(H,u)$ and an isomorphism of lattice-ordered monoids $\eta_I: I \lto G_+$.
\item Let $(H,u)\cong \Gamma^{-1}(A/I) \text{  and  }  G\cong \Delta^{-1}(\langle I\rangle)$
\end{itemize}

The function $f_I: A \lto \Gamma(H \lex G, (u,0))$ defined by
$$f_I(a)=\left(\zeta_I(s_a),\eta_I(\varepsilon_a)-\eta_I(\tau_a)\right), \text{ for any }a\in A$$ is an isomorphism of MV-algebras.
\end{itemize}
\end{proof}

\begin{corollary}\label{rad A es lex iff}
If $A$ is a lexicographic MV-algebra the following are equivalent:
\begin{enumerate}[(1)]
\item $\Rad A\in \Lex\Id(A)$,
\item there exists an $\ell u$-subgroup $(R',1)$ of $(\mathbb{R},1)$ and a non-trivial $\ell$-group $G$ such that $$A\cong\Gamma(R'\lex G,(1,0)).$$
\end{enumerate}
Moreover, if the above equivalent conditions are satisfied the $\ell u$-subgroup $(R',1)$ of $(\mathbb{R},1)$ and the $\ell$-group $G$ are uniquely determined, up to isomorphisms.
\end{corollary}

In \cite{dfl}, the authors also showed the following inclusions which give an interesting classification of some classes of MV-algebras:

\begin{center}
Perfect $\subset$ Local with retractive radical $\subset$ Lexicographic $\subset$ Local.
\end{center}

\section{MV-Topological Spaces}\label{mvtopsec}

Both crisp and fuzzy subsets of a given set will be identified with their membership functions and usually denoted by lower case latin or greek letters. In particular, for any set $X$, we shall use also $\1$ and $\0$ for denoting, respectively, $X$ and $\varnothing$. In some cases, we shall use capital letters in order to emphasize that the subset we are dealing with is crisp.

An MV-topological space is basically a special fuzzy topological space in the sense of C. L. Chang \cite{chal} and most of the definitions and results of the present subsection are simple adaptations of the corresponding ones of the aforementioned work to the present context or directly derivable from the same work or from the results presented in the papers \cite{hoh1,hoh2,hoh3,low,rod1,rod2,sto1,sto2}.

\begin{definition}\label{mvtop}
Let $X$ be a set, $A$ the MV-algebra $[0,1]^X$ and $\tau \subseteq A$. We say that $\left( X, \tau\right)$ is an \emph{MV-topological space} (or \emph{MV-space}) if $\tau$ is a subuniverse both of the quantale $\left( [0,1]^X, \bigvee, \oplus\right)$ and of the semiring $\left( [0,1]^X, \wedge, \odot, \1\right)$. More explicitly, $\left( X, \tau\right)$ is an MV-topological space if
\begin{enumerate}[(i)]
\item $\0, \1 \in \tau$,
\item for any family $\{o_i\}_{i \in I}$ of elements of $\tau$, $\bigvee_{i \in I} o_i \in \tau$,
\end{enumerate}
and, for all $o_1, o_2 \in \tau$,
\begin{enumerate}[(i)]
\setcounter{enumi}{2}
\item $o_1 \odot o_2 \in \tau$,
\item $o_1 \oplus o_2 \in \tau$,
\item $o_1 \wedge o_2 \in \tau$.
\end{enumerate}
$\tau$ is also called an \emph{MV-topology} on $X$ and the elements of $\tau$ are the \emph{open MV-subsets} of $X$. The set $\tau^* = \{o^* \mid o \in \tau\}$ is easily seen to be a subquantale of $\left( [0,1]^X, \bigwedge, \odot\right)$ (where $\bigwedge$ has to be considered as the join w.r.t. to the dual order $\geq$ on  $[0,1]^X$) and a subsemiring of $\left( [0,1]^X, \vee, \oplus, \0\right)$, i.e., it verifies the following properties:
\begin{itemize}
\item[$-$] $\0, \1 \in \tau^*$,
\item[$-$] for any family $\{c_i\}_{i \in I}$ of elements of $\tau^*$, $\bigwedge_{i \in I} c_i \in \tau^*$,
\item[$-$] for all $c_1, c_2 \in \tau^*$, $c_1 \odot c_2, c_1 \oplus c_2, c_1 \vee c_2 \in \tau^*$.
\end{itemize}
The elements of $\tau^*$ are called the \emph{closed MV-subsets} of $X$.
\end{definition}

Let $X$ and $Y$ be sets. Any function $f: X \lto Y$ naturally defines a map
\begin{equation}\label{muf}
\begin{array}{cccc}
f\fcou: & [0,1]^Y & \lto 		& [0,1]^X \\
		 & \alpha  & \lmapsto & \alpha \circ f.
\end{array}
\end{equation}
Obviously $f\fcou(\0) = \0$; moreover, if $\alpha, \beta \in [0,1]^Y$, for all $x \in X$ we have $f\fcou(\alpha \oplus \beta)(x) = (\alpha \oplus \beta)(f(x)) = \alpha(f(x)) \oplus \beta(f(x)) = f\fcou(\alpha)(x) \oplus f\fcou(\beta)(x)$ and, analogously, $f\fcou(\alpha^*) = f\fcou(\alpha)^*$. Then $f\fcou$ is an MV-algebra homomorphism and we shall call it the \emph{MV-preimage} of $f$. The reason of such a name is essentially the fact that $f\fcou$ can be seen as the preimage, via $f$, of the fuzzy subsets of $Y$. From a categorical viewpoint, once denoted by $\Set$, $\BB$ and $\MV$ the categories of sets, Boolean algebras, and MV-algebras respectively (with the obvious morphisms), there exist two contravariant functors $\wp: \Set \lto \BB^{\op}$ and $\fwp: \Set \lto \MV^{\op}$ sending each map $f: X \lto Y$, respectively, to the Boolean algebra homomorphism $f\cou: \wp(Y) \lto \wp(X)$ and to the MV-homomorphism $f\fcou: [0,1]^Y \lto [0,1]^X$.

Moreover, for any map $f: X \lto Y$ we define also a map $f^\to: [0,1]^X \lto [0,1]^Y$ by setting, for all $\alpha \in [0,1]^X$ and for all $y \in Y$,
\begin{equation}\label{ovf}
f^\to(\alpha)(y) = \bigvee_{f(x) = y} \alpha(x).
\end{equation}
Clearly, if $y \notin f[X]$, $f^\to(\alpha)(y) = \bigvee \varnothing = \0$ for any $\alpha \in [0,1]^X$.

\begin{definition}\cite{chal}\label{cont}
Let $\left( X, \tau_X \right)$ and $\left( Y, \tau_Y \right)$ be two MV-topological spaces. A map $f: X \lto Y$ is said to be
\begin{itemize}
\item \emph{continuous} if $f\fcou[\tau_Y] \subseteq \tau_X$,
\item \emph{open} if $f^\to(o) \in \tau_Y$ for all $o \in \tau_X$,
\item \emph{closed} if $f^\to(c) \in \tau^*_Y$ for all $c \in \tau^*_X$
\item an \emph{MV-homeomorphism} if it is bijective and both $f$ and $f^{-1}$ are continuous.
\end{itemize}
\end{definition}
We can use the same words of the classical case because, as it is trivial to verify, if a map between two classical topological spaces is continuous, open, or closed in the sense of the definition above, then it has the same property in the classical sense.

\begin{definition}\cite{war}\label{base}
As in classical topology, we say that, given an MV-topological space $\left( X, \tau\right)$, a subset $B$ of $[0,1]^X$ is called a \emph{base} for $\tau$ if $B\subseteq \tau$ and every open set of $(X,\tau)$ is a join of elements of $B$.
\end{definition}

\begin{lemma}\cite{rus2}\label{bascont}
Let $\left( X, \tau_X\right)$ and $\left( Y, \tau_Y\right)$ be two MV-topological spaces and let $B$ be a base for $\tau_Y$. A map  $f: X \lto Y$ is continuous if and only if $f\fcou[B] \subseteq \tau_X$.
\end{lemma}

A \emph{covering} of $X$ is any subset $\Gamma$ of $[0,1]^X$ such that $\bigvee \Gamma = \1$ \cite{chal}, while an \emph{additive covering} ($\oplus$-covering, for short) is a finite family $\{\alpha_i\}_{i=1}^n$ of elements of $[0,1]^X$, $n < \omega$, such that $\alpha_1 \oplus \cdots \oplus \alpha_n = \1$. It is worthwhile remarking that we used the expression ``finite family'' in order to include the possibility for such a family to have repetitions. In other words, an additive covering is a finite subset $\{\alpha_1, \ldots, \alpha_k\}$ of $[0,1]^X$, along with natural numbers $n_1, \ldots, n_k$, such that $n_1\alpha_1 \oplus \cdots \oplus n_k \alpha_k = \1$.

\begin{definition}\label{compact}
An MV-topological space $\left( X, \tau \right)$ is said to be \emph{compact} if any open covering of $X$ contains an additive covering; it is called \emph{strongly compact} if any open covering contains a finite covering.\footnote{What we call strong compactness here is called simply compactness in the theory of lattice-valued fuzzy topologies \cite{chal}.}
\end{definition}

\begin{definition}\label{t2ax}
Let $\left( X, \tau\right)$ be an MV-topological space. $X$ is called a \emph{Hausdorff} (or \emph{separated}) \emph{space} if, for all $x \neq y \in X$, there exist $o_x, o_y \in \tau$ such that
\begin{enumerate}[(i)]
\item $o_x(x) = o_y(y) = 1$,
\item $o_x \wedge o_y = \0$.
\end{enumerate}
\end{definition}

\section{The Maximal Spectrum and MV-Spectrum of an MV-algebra}\label{specs}

In the present section, we shall recall the Zariski topology on the set $\Max A$ of maximal ideals of an MV-algebra $A$; then we shall see the MV-topology defined in \cite{rus2} on the same set and how the two spaces are related to each other. In order to do that, let us first consider the set $\Spec A$ of all prime ideals of $A$.

For any ideal $I$ of $A$, let
\begin{equation}\label{r(I)}
  r(I) = \{P\in \Spec A: I\nsubseteq P\}
\end{equation}
Then the set $\tau = \{r(I): I\in \Id(A)\}$ is the family of open set of a topology on $\Spec A$. Indeed,
\begin{enumerate}[(i)]
  \item $r(\{\0\})=\emptyset$,
  \item $r(A)=\Spec A$,
  \item $r(I\wedge J)=r(I)\cap r(J)$ for all $I,J\in \Id(A)$,
  \item $r(\bigvee\{I_\lambda:\lambda\in \Lambda\})=\bigcup\{r(I_\lambda):\lambda\in\Lambda\}$ for any $\{I_\lambda:\lambda\in \Lambda\}\subseteq\Id(A).$
\end{enumerate}
In the sequel, $\tau$ will be referred to as the \emph{spectral topology} or the \emph{Zariski topology}.

Now, for any $a\in A$, let
\begin{equation}
r(a) = \{P\in \Spec A:a\notin P\}.
\end{equation}
We have the following properties.
\begin{lemma}\label{rprop}\emph{\cite{dinleusbook}}
\begin{enumerate}[(i)]
\item $r(a)=r((a])$ for any $a\in A$,
\item $r(\0)=\emptyset$,
\item $r(\1)=\Spec A$,
\item $r(a\vee b)=r(a\oplus b)=r(a)\cup r(b)$, for all $a, b \in A$,
\item $r(a\wedge b)=r(a)\cap r(b)$, for all $a, b \in A$,
\item $r(I)=\bigcup\{r(a):a\in I\}$, for any $I\in \Id(A)$.
\end{enumerate}
\end{lemma}

By Lemma \ref{rprop}(i,vi), $\{r(a):a\in A\}$ is a basis for the topology $\tau$. It is well-known also that the compact open subsets of $\Spec A$ are exactly the sets of the form $r(a)$ for some $a\in A$. In particular, $\Spec A$ is compact because $r(\1)=\Spec A$ (see \cite{dinleusbook}).

For each $a\in A$, the set $$H(a):=\{P\in\Spec A: a\in O_P\}$$ is an open set of $\Spec A$ \cite[Lemma 3.6]{ferraleti}.

Since $\Max A \subseteq\Spec A$ we can endow $\Max A$ with the topology induced by the spectral topology $\tau$ on $\Spec A$. This means that the open sets of $\Max A$ are
$$R(I)=r(I)\cap\Max A=\{M\in \Max A:I\nsubseteq M\}$$
So, for any $a\in A$ and $I\in\Id A$
$$R(a)=r(a)\cap\Max A=\{M\in \Max A:a\notin M\} \text{ and } R(I)=\bigcup\{R(a):a\in I\}$$
Hence the family $\{R(a):a\in A\}$ is a basis for the induced topology on $\Max A$. The set of opens in $\Max A$ will be denoted by $\mathcal{O}(\Max A)$.

By \cite[Theorem 3.6.10]{dinleusbook}, we have that for any MV-algebra $A$ the maximal ideal space, $\Max A$, is a compact Hausdorff topological space with respect to the topology induced by the spectral topology on $\Spec A$.

%Analogously, we have that the open sets of $\Min A$ are $$d(I)=r(I)\cap\Min A=\{m\in \Min A:I\nsubseteq m\}.$$

It is very well-known \cite{bel,cha1} that, for any MV-algebra $A$, there exists a canonical homomorphism $\iota: A \to [0,1]^{\Max A}$, where $\Max A$ is the set of maximal ideals of $A$. Such a homomorphism is defined as follows:
\begin{itemize}
\item for each $M \in \Max A$, there is the natural projection $\pi_M: A \lto A/M$;
\item for any $M \in \Max A$, $A/M$ is a simple MV-algebra and, therefore, is isomorphic to a subalgebra of $[0,1]$, i.e., there exists a (unique) embedding $\iota_M: A/M \lto [0,1]$;
\item the morphism $\iota: A \lto [0,1]^{\Max A}$ associates, to each $a \in A$, the fuzzy subset $\widehat a$ of $\Max A$ defined by $\widehat a(M) = \iota_M(\pi_M(a)) = \iota_M(a/M)$ for all $M \in \Max A$.
\end{itemize}
The kernel of $\iota$ is exactly $\Rad A$, and the homomorphism $\iota$ is an embedding if and only if $A$ is a semisimple algebra. So, for any MV-algebra $A$, $A/Rad A$ is isomorphic to a subalgebra $A'$ of $[0,1]^{\Max A}$. Therefore, $A'$ is a covering of $\Max A$ and, since it is a subalgebra of $[0,1]^{\Max A}$, it is closed under $\oplus$, $\odot$ and $\wedge$. Then it is a base for an MV-topology on $\Max A$. In the following results we shall often identify any semisimple MV-algebra $A$ with its isomorphic image included in $[0,1]^{\Max A}$; so any element $a$ of a semisimple MV-algebra will be identified with the fuzzy set $\widehat a$. The reader may refer to \cite{bel,cha1,cha2,mvbook} for further details.

\begin{definition}\cite{rus2}\label{mvspec}
The \emph{maximal MV-spectrum} of $A$ is the MV-topology $\tau_A$ on $\Max A$ whose base is the image $A'$ of the morphism $\iota$.
\end{definition}

%Let $(\Max A,\tau_A)$ be the Maximal MV-Spectrum of $A$. Let us see some of its properties and its relation with the topological space $\Max A$ with the Zariski topology. The basic opens of $\Max A$ denoted by $R(a)$ with $a\in A$, were defined in Subsection \ref{spectop}.

\begin{proposition}\label{r(a)=supp(widehat(a)}
Let $A$ be an MV-algebra and $(\Max A,\tau_A)$ be the associated MV-topological space. For each basic fuzzy open $\widehat{b}\in \tau_A$, $R(b)=\supp(\widehat{b})$.

Consequently, the Zariski topology on $\Max A$, as an MV-topology, is coarser than $\tau_A$.
\end{proposition}
\begin{proof}
  In fact, for each $M\in\Max A$, $\widehat{b}(M)=\frac{b}{M}=0$ if and only if $b\in M$. That is, $M\in\supp \widehat{b}$ iff $\widehat{b}(M)=\frac{b}{M}>0$ iff $b\notin M$ iff $M\in R(b)$. The second statement follows from the first one and \cite[Proposition 3.5]{dlpr}
\end{proof}

\begin{proposition}\label{H(a) is open}
  For each $a\in A$, the set $H(a)=\{M\in\Max A: a\in O_M\}$ is an element of $\tau_A$.
\end{proposition}
\begin{proof}
We will prove that $H(a)$ is the support of a fuzzy open of $\tau_A$. If $M\in H(a)$ then $a\in O_M$, so by Proposition \ref{O_P ch} there exists $b_M\notin M$ such that $a\wedge b_M=\0$. That is, $M\in R(b_M)=\supp\left(\widehat{b_M}\right)$. Let us see that
	$$H(a)=\supp\left(\bigvee_{M\in H(a)}\widehat{b_M}\right).$$
  In fact, if $N\in H(a)$ then there exists $b_N$ such that $b_N\notin N$ and $a\wedge b_N=\0$, then $\widehat{b_N}(N)>0$, and therefore $\left(\bigvee_{M\in H(a)}\widehat{b_M}\right)(N)= \bigvee_{M\in H(a)}\widehat{b_M}(N) >0$, i.e, $N\in \supp\left(\bigvee_{M\in H(a)}\widehat{b_M}\right)$. For the other inclusion, if $\left(\bigvee_{M\in H(a)}\widehat{b_M}\right)(N)>0$ then there exists $\widehat{b_M}$ with $M\in H(a)$ such that $\widehat{b_M}(N)>0$, i.e., $b_M\notin N$ and $a\wedge b_M=\0$, then $a\in O_N$ and therefore $N\in H(a)$.
\end{proof}

\section{MV-sheaves}\label{sheafsec}

Let $\left(X,\tau \right)$ be an MV-topological space. The poset of open fuzzy subsets $\tau\subseteq [0,1]^X$, with the fuzzy inclusion $\leq$, can be viewed as a category in the usual manner, namely, $\tau$ is the object class and, for all $\alpha,\beta\in \tau$, there is exactly one morphism $\alpha\lto \beta$ if $\alpha\leq \beta$, there are none otherwise.

%ver Tennison

\begin{definition}
  Let $(X,\tau)$ be an MV-topological space and let $\C$ be a category. % (of algebras).
	An \emph{MV-presheaf} of $\obj(\C)$ on $X$ is a contravariant functor $F:\tau\lto \C$, that is:

  \begin{enumerate}[(i)]
    \item for each fuzzy open set $\alpha$ in $\tau$, $F(\alpha)$ is an object of $\C$, called the set of sections of $F$ over $\alpha$;
    \item for each pair of fuzzy open sets $\beta\leq\alpha$ in $\tau$, the image of the morphism $\beta \lto \alpha$ is the so-called \emph{restriction map} $\rho_\beta^\alpha:F(\alpha)\lto F(\beta)$ with the following properties:
    \begin{enumerate} [(a)]
      \item $\rho_\alpha^\alpha=\id_{F(\alpha)}$, for all $\alpha$;
      \item $\rho_\gamma^\alpha=\rho_\gamma^\beta\circ\rho_\beta^\alpha$, whenever $\gamma\leq\beta\leq\alpha$ in $\tau$.
    \end{enumerate}
  \end{enumerate}
\end{definition}

\begin{definition}
    Let $F$ and $G$ be MV-presheaves of $\ob(\C)$ over $(X,\tau)$. A \emph{morphism of MV-presheaves} from $F$ to $G$ is a natural transformation $f:F\To G$, that is, a family $\{f(\alpha):F(\alpha)\lto G(\alpha)\}_{\alpha\in \tau}$ such that, whenever $\beta\leq\alpha$ are open fuzzy sets in $\tau$, the diagram
  $$\xymatrix{
F(\alpha) \ar[r]^{f(\alpha)} \ar @{->}[d]_{\rho_\beta^\alpha} & G(\alpha)  \ar @{->}[d]^{\rho\prime_\beta^\alpha}\\
F(\beta) \ar[r]^{f(\beta)}               		& G(\beta)	
}$$
commutes.
\end{definition}

\begin{exm}\label{const}
Let $A$ be a fixed object in the category $\C$ and $(X,\tau_X)$ be an MV-space. We define the constant MV-presheaf $A_X:\tau_X\lto \C$ on $(X,\tau_X)$, by setting:
\begin{itemize}
\item $A_X(\alpha)=A$ for all $\alpha$ in $\tau_X$, and
\item $\rho_\beta^\alpha=\id_A: A_X(\alpha)\lto A_X(\beta)$ for $\beta\leq\alpha$ in $\tau_X$.
\end{itemize}
\end{exm}
\begin{exm}\label{CY}
Let $(X,\tau_X)$ and $(Y, \tau_Y)$ be MV-topological spaces. Let us consider $C^Y:\tau_X\to\Set$ defined by $$C^Y(\alpha)=\{f:\supp(\alpha)\lto Y\mid f \text{ is continuous}\},$$ with $\rho_\beta^\alpha: C^Y(\alpha)\lto C^Y(\beta)$ such that $\rho_\beta^\alpha(f)=f_{\restr\supp(\beta)}$ for $\beta\leq\alpha$ in $\tau_X$. $C^Y$ is an MV-presheaf of sets over $X$. Note that $\supp(\beta)\subseteq \supp(\alpha)$ if $\beta\leq\alpha$.
\end{exm}

\begin{definition}\label{sheaf}
  An MV-presheaf of sets over the MV-topological space $(X,\tau_X)$ satisfying the following two conditions is called an \emph{MV-sheaf} of $\ob(\C)$.
  \begin{enumerate}[(i)]
    \item If $\alpha$ is a fuzzy open set of $X$ and the family $\{\alpha_i\}_{i\in I}\subseteq [0,1]^X$ is an open covering of $\alpha$, i.e., $\alpha = \bigvee_{i \in I} \alpha_i$, and $s, s' \in F(\alpha)$ are two sections of $F$ such that for all $i \in I$
        $$\rho_{\alpha_i}^\alpha(s)=\rho_{\alpha_i}^\alpha(s')$$
        then $s=s'$.
    \item If $\alpha$ is a fuzzy open set of $X$ and the family $\{\alpha_i\}_{i\in I}\subseteq [0,1]^X$ is an open covering of $\alpha$; and if there is a family $\{s_i\}_{i\in I}$ of sections of $F$ with $s_i\in F(\alpha_i)$ for all $i\in I$, such that for all $i,j\in I$
$$\rho_{\alpha_i\wedge\alpha_j}^{\alpha_i}(s_i)=\rho_{\alpha_i\wedge\alpha_j}^{\alpha_j}(s_j)$$
then there is $s\in F(\alpha)$ such that for all $i\in I$
$$\rho_{\alpha_i}^\alpha(s)=s_i.$$
  \end{enumerate}
  In other words, if the system $(s_i)_{i\in I}$ is given on a covering and is consistent on all of the overlaps, then it comes from a section over all of the $\alpha$'s.
  \end{definition}

\begin{definition}
  If $F, G$ are MV-sheaves of $\ob(\C)$ and $f : F\To G$ is an MV-presheaf morphism, we also call $f$ a \emph{morphism of MV-sheaves}.
\end{definition}

\begin{exm}
  The MV-presheaf $C^Y$, described in the Example \ref{CY}, is an MV-sheaf. Let us see that $C^Y$ satisfies the two conditions Definition \ref{sheaf}.
	
  Let $\alpha$ be a fuzzy open set of $X$ and let $\{\alpha_i\}_{i\in I}\subseteq [0,1]^X$ be an open covering of $\alpha$, i.e., $\alpha = \bigvee_{i \in I} \alpha_i$,
  \begin{enumerate}[(i)]
    \item let $f, f' \in C^Y(\alpha)$ be two sections of $C^Y$ such that for all $i \in I$, $$\rho_{\alpha_i}^\alpha(f)=\rho_{\alpha_i}^\alpha(f'),$$
        that is, $$f_{\restr\supp(\alpha_i)}=f'_{\restr\supp(\alpha_i)}$$ where $f,f': \supp(\alpha)\to Y$.
				
        Note that $\bigcup_{i\in I}\supp(\alpha_i)=\supp(\alpha)$ because $\alpha = \bigvee_{i \in I} \alpha_i$. Let us see that $f=f'$.
				
        If $x\in \supp(\alpha)$, then there exists $i\in I$ such that $x\in \supp(\alpha_i)$, so $$f(x)=f_{\restr\supp(\alpha_i)}(x)=f'|\supp(\alpha_i)(x)=f'(x)$$ then $f=f'.$

    \item For the second condition, suppose that there is a family $\{f_i\}_{i\in I}$ of sections of $C^Y$ with $f_i\in C^Y(\alpha_i)$ for all $i\in I$, such that for all $i,j\in I$
$$\rho_{\alpha_i\wedge\alpha_j}^{\alpha_i}(f_i)=\rho_{\alpha_i\wedge\alpha_j}^{\alpha_j}(f_j)$$
We define $f:=\bigcup_{i\in I}f_i: \supp(\alpha)\lto Y$ by $f(x)=f_i(x)$ if $x\in \supp(\alpha_i)=\dom(f_i)$. We know that $\supp(\alpha)=\bigcup_{i\in I}\supp(\alpha_i)$, then $f$ is well defined because for all $i,j\in I$, $x\in \supp(\alpha_i)\cap \supp(\alpha_j)$ iff $x\in\supp(\alpha_i\wedge\alpha_j)$, and by hypothesis $${f_i}_{\restr\supp(\alpha_i\wedge\alpha_j)}(x)={f_j}_{\restr\supp(\alpha_i\wedge\alpha_j)}(x)$$ where $f_i:\supp(\alpha_i)\lto Y$ and $f_j: \supp(\alpha_j)\lto Y$. It is clear that $f_{\restr\supp(\alpha_i)}=f_i$, for each $i\in I$.

Now, let us prove that $f$ is continuous.

Let $\gamma\in \tau_Y$, and let us prove that $\gamma\circ f\in \tau_{\supp(\alpha)}$. For each $i\in I$, $\gamma\circ f_i\in \tau_{\supp(\alpha_i)}$, i.e., $\gamma\circ f_i=\beta\wedge\supp(\alpha_i)$ with $\beta\in\tau_X$. As $\supp(\alpha_i)=\supp(\alpha_i)\wedge\supp(\alpha)$, then $\gamma\circ f_i=\beta\wedge\supp(\alpha_i)\wedge\supp(\alpha)$. Thus, for each $i\in I$, $\gamma\circ f_i\in\tau_{\supp(\alpha)}$ because $\beta\wedge\supp(\alpha_i)\in\tau_X$. Therefore, $\gamma\circ f=\bigvee_{i\in I}(\gamma\circ f_i)\in \tau_{\supp(\alpha)}$.
\end{enumerate}
\end{exm}

%\textbf{Aquí}: por que un abierto restringido a un conjunto menor, visto como abierto en un subespacio, resulta ser abierto tambien en el espacio grande? poniendo cero en el complemento? que me asegura que sea abierto en el  maas grande?

\begin{definition}
 A \emph{directed set} $I$ is a set with a pre-order $\leq$ which satisfies the following:
 \begin{enumerate}
   \item[(a)] for all $i,j\in I$, there exists $k\in I$ such that $i\leq k$ and $j\leq k$.
  \end{enumerate}
  A \emph{direct system} of sets indexed by a directed set $I$ is a family $\{\alpha_i\}_{i\in I}$ of sets together with maps $\rho_{ij}:\alpha_i\lto\alpha_j$, for each $i\leq j \in I$, satisfying
  \begin{enumerate}
    \item[(b)] For all $i\in I$, $\rho_{ii}=\id_{\alpha_i}$;
    \item[(c)] For all $i,j,k\in I$, $i\leq j\leq k$ implies $\rho_{i k}=\rho_{j k}\circ\rho_{i j}$.
  \end{enumerate}
\end{definition}

Let $F$ be an MV-presheaf of $\ob(\C)$ over an MV-topological space $(X,\tau)$ and fix $x\in X$. Then $\{F(\alpha): x\in \supp(\alpha)\}$,  forms a direct system with maps $\rho_\beta^\alpha:F(\alpha)\lto F(\beta)$, whenever $\beta\leq\alpha$, and $x\in \supp(\beta)\subseteq \supp(\alpha)$. We have the following definition:

\begin{definition}
The \emph{MV-stalk} $F_x$ of $F$ at $x$ is $$\lim_{x\in \supp(\alpha)} F(\alpha),$$
which comes equipped with maps $F(\alpha)\lto F_x$ such that $s\lmapsto s_x$ whenever $x\in \supp(\alpha)$ for $\alpha \in \tau$. The members of $F_x$ are also called \emph{germs} (of sections of $F$).
\end{definition}
%\begin{proposition}(tennison pag 9)
 % \begin{enumerate}[(a)]
  %  \item Each germ $t\in F_x$ arises as $t=s_x$ for some $s\in F(\alpha)$ for open neighbourhood $\alpha$ of $x$.
%    \item Two germs $s_x,t_x\in F_x$, with $s\in F(\alpha), t\in F(\beta)$, are equal, i. e., $s_x=t_x$ iff there exist an open fuzzy set $\gamma\leq\alpha\wedge\beta$ such that $\rho_\gamma^\alpha(s)=\rho_\gamma^\beta(t)$.
  %\end{enumerate}
%\end{proposition}

\begin{definition}
  Let $(X,\tau_X)$ be an MV-topological space. An \emph{MV-sheaf space} over X is a triple $(E,p,X)$ where $(E,\tau_E)$ is an MV-topological space and $p:E\lto X$ is a \emph{local MV-homeomorphism}, that is, $p$ is continuous and, for all $x\in E$, there exists an open fuzzy set $\alpha\in\tau_E$ such that $\alpha(x)>0$ and an open fuzzy set $\beta\in \tau_X$ such that $p_{\restr\supp(\alpha)}: \supp(\alpha)\lto \supp(\beta)$ is an MV-homeomorphism.
  \end{definition}
A morphism of MV-sheaf spaces over $X$, $f: (E, p, X)\lto (E',p',X)$, is a continuous map $f:E\lto E'$ such that $p=p'\circ f$.

We can construct an MV-sheaf of sets from an MV-sheaf space and reciprocally, we can construct an MV-sheaf space from an MV-sheaf. These constructions follow the canonical rules of sheaf theory on topological spaces (see \cite{davey, tenn}).

\section{Locally retractive MV-algebras}
\label{mvrsec}

In this section we will define the class of locally retractive MV-algebras, which are the algebras that will be isomorphically represented by MV-sheaves of $\ell$-groups, and we shall discuss some properties of such algebras, along with their relationship with lexicographic MV-algebras. We recall that an MV-algebra $A$ is said to have \emph{retractive radical} if the natural projection $p: A \to A/\Rad A$ is a retraction, i.e., if there exists an embedding $j: A/\Rad A \to A$ such that $p \circ j = \id_{A/\Rad A}$. We shall prove that algebras with retractive radical are locally retractive (Theorem \ref{mvrthm}) while the converse is not true (Example \ref{mvlexm}). However, we will also give a necessary and sufficient condition under which a locally retractive MV-algebra has retractive radical (Theorem \ref{mvlthm}).

\begin{definition}\label{mvl}
An MV-algebra $A$ is called \emph{locally retractive} if, for all $M \in \Max A$, $A/O_M$ has retractive radical. We shall denote by $\MVl$ the full subcategory of $\MV$ whose objects are the locally retractive algebras.
\end{definition}

Before proving the next results, we recall that, for any MV-algebra $A$ and for each $M \in \Max A$, $\Rad(A/O_M) = M/O_M$ and, therefore, $\frac{A/O_M}{\Rad(A/O_M)} \cong A/M$.

\begin{lemma}\label{prodmvl}
If $A \cong \prod_{M \in \Max A} A/O_M$, then $A \in \MVl$ if and only if it has retractive radical.
\end{lemma}
\begin{proof}
If $A$ is locally retractive let, for each $M \in \Max A$, $j_M: A/M \to A/O_M$ be the right-inverse to the natural projection $p_M$, and let us consider the following diagram:
\begin{equation}\label{diag}
\xymatrix{
A/\Rad A \ar@{^{(}->}[dd]_{i'} \ar@{_{(}-->}@<.8ex>[rr]^{j} & & A \ar@{->>}@<.8ex>[ll]^{p} \ar@{_{(}->}[dd]^{i} \\
 &&\\
\prod\limits_{M \in \Max A} A/M \ar@{^{(}->}@<-.8ex>[rr]_{j'} & & \prod\limits_{M \in \Max A} A/O_M \ar@{->>}@<-.8ex>[ll]_{p'}
}
\end{equation}
where $i$ and $i'$ are the canonical subdirect embeddings -- which are actually isomorphisms in this case, by hypothesis -- and
$$j'\left((a/M)_{M \in \Max A}\right) = \left((j_M(a/M)\right)_{M \in \Max A}, \quad p'\left((a/O_M)_{M \in \Max A}\right) = \left((p_M(a/O_M)\right)_{M \in \Max A}.$$

Then the embedding of $A/\Rad A$ into $A$ is given by
$$j: a/\Rad A \in A/\Rad A \mapsto (i^{-1} \circ j' \circ i')(a/\Rad A) \in A$$.

Conversely, with reference again to diagram (\ref{diag}) and assuming that $j$ exists, let us observe that, for all $M \in \Max A$,
\begin{itemize}
\item $i[M] = M/O_M \times \prod\limits_{N \in \Max A \setminus \{M\}} A/O_N$,
\item $\Rad \left(\prod\limits_{N \in \Max A} A/O_N\right) = i[\Rad A] = \prod\limits_{N \in \Max A} N/O_N$, and
\item $i[O_M] = \{0\} \times \prod\limits_{N \in \Max A \setminus \{M\}} A/O_N$.
\end{itemize}
Then it is clear that, with our hypotheses, we get $i[M] = i[O_M] \oplus i[\Rad A]$ and, therefore, $M = O_M \oplus \Rad A$. Now, for each maximal ideal $M$, let
\begin{equation}\label{jm}
j_M: a/M \in A/M \mapsto \frac{j(a/\Rad A)}{O_M} \in A/O_M.
\end{equation}
For all $a,b \in A$ and $M \in \Max A$, we have:
$$\begin{array}{l}
a/M = b/M \iff \frac{a}{O_M \oplus \Rad A} = \frac{b}{O_M \oplus \Rad A} \iff \\
 \\
\frac{a/\Rad A}{O_M/\Rad A} = \frac{b/\Rad A}{O_M/\Rad A} \iff  \frac{j(a/\Rad A)}{O_M} = \frac{j(b/\Rad A)}{O_M},
\end{array}$$
whence each $j_M$ is well-defined and injective. Since it is obviously a homomorphism, it follows that $A/O_M$ has retractive radical for all $M \in \Max A$, i.e., $A \in \MVl$.
\end{proof}

\begin{theorem}\label{mvrthm}
If an MV-algebra $A$ has retractive radical, then it is locally retractive.
\end{theorem}
\begin{proof}
We shall prove that $\prod_{M \in \Max A} A/O_M$ has retractive radical, then the assertion will follow from Lemma \ref{prodmvl}. Let us refer again to diagram (\ref{diag}), but keeping in mind that now $i$ and $i'$ are subdirect embeddings but not necessarily isomorphisms, and that $j$ exists, while we want to prove that $j'$ exists too. For all $(a_M/O_M)_{M \in \Max A} \in \prod_{M \in \Max A} A/O_M$, and for all $M \in \Max A$, since $i$ is a subdirect embedding, there exists $a_M' \in A$ such that $\frac{j(a_M'/\Rad A)}{O_M} = a_M/M$.

We set
$$j': (a_M/M)_{M \in \Max A} \in \prod_{M \in \Max A} A/M \mapsto \left(\frac{j(a_M'/\Rad A)}{O_M}\right)_{M \in \Max A} \in \prod_{M \in \Max A} A/O_M,$$
and we have
$$\begin{array}{l}
(a_M/M)_{M \in \Max A} = (b_M/M)_{M \in \Max A} \iff \\
\forall M \in \Max A \ (a_M/M = b_M/M) \iff  \\
\forall M \in \Max A \ \left(\frac{j(a_M'/\Rad A)}{O_M} = \frac{j(b_M'/\Rad A)}{O_M}\right) \iff \\
j'\left((a_M/M)_{M \in \Max A}\right) = j'\left((b_M/M)_{M \in \Max A}\right).
\end{array}$$
It follows that $j'$ is a well-defined injective map; since it is obviously a homomorphism too, then $\prod_{M \in \Max A} A/O_M$ has retractive radical, whence, $A/O_M$ has retractive radical for all $M \in \Max A$. The assertion is proved.

\end{proof}

The following example shows that the converse of Theorem \ref{mvrthm} does not hold.

\begin{exm}\label{mvlexm}
Let $K_3$ be the Komori chain of rank 3, i.e., $K_3 = \Gamma(\Z \lex \Z, (2,0))$, and let $X$ be an arbitrary infinite set. $K_3$ has obviously retractive radical as well as $K_3^X$ by an easy application of Lemma \ref{prodmvl}. Observe that the universe of $K_3$ is the set $(\{0\} \times \N) \cup (\{1\} \times \Z) \cup (\{2\} \times -\N)$, where $-\N$ is the set of non-positive integers. Let us also denote by $\O$ and $\E$, respectively, the sets of odd and even integers.

Let us consider the subset $A$ of $K_3^X$ defined in the following way:
$$\forall a \in K_3^X \ (a \in A \iff a^{-1}(\{0\} \times (\N \cap \O)) \cup a^{-1}(\{1\} \times (\E)) \cup a^{-1}(\{2\} \times (-\N \cap \O)) \text{ is finite}).$$
It is easy to verify that $A$ is a subalgebra of $K_3^X$, but $A/\Rad A \cong \{0, 1/2, 1\}^X$ is not embeddable in $A$.
\end{exm}

With reference to the diagram (\ref{diag}), we can prove
\begin{theorem}\label{mvlthm}
A locally retractive algebra $A$ has retractive radical if and only if
$$(j' \circ i')[A/\Rad A] \subseteq i[A].$$
\end{theorem}
\begin{proof}
If $A$ has retractive radical, then the arrow $j$ in (\ref{diag}) exists, and $j' \circ i' = i \circ j$, hence
$$(j' \circ i')[A/\Rad A] = (i \circ j)[A/\Rad A] \subseteq i[A].$$

Reciprocally, if $(j' \circ i')[A/\Rad A] \subseteq i[A]$, then is well-defined the map
$$j: a/\Rad A \in A/\Rad A \mapsto i^{-1}((j' \circ i')(a/\Rad A)) \in A,$$
and it is easily seen to be an MV-algebra embedding.
\end{proof}

We conclude this section with a technical lemma that will be used in the next section.

\begin{lemma}\label{lexrepr}
Let $A \in \MVl$ and, for all $M \in \Max A$, let $G(M/O_M)$ be the $\ell$-group completion of the cancellative lattice-ordered monoid $\la M/O_M, \oplus, 0\ra$. Then, for all $M \in \Max A$, the mapping
$$a/O_M \in A/O_M \mapsto (\widehat{a}(M),g_{aM}) \in \Gamma(R_M\lex G(M/O_M),(1,0)),
$$
where  $g_{aM}=(\frac{a}{O_M} \ominus \frac{a}{M})-(\frac{a}{M}\ominus \frac{a}{O_M}) \in G(M/O_M)$, and $\la R_M, 1\ra = \Gamma^{-1}(A/M) \leq \la \R, 1\ra$, is an MV-algebra isomorphism.
\end{lemma}
\begin{proof}
The statement follows readily from the definitions and Theorem \ref{reprthlex}. Indeed, as we already observed, since $A/O_M$ is a local MV-algebra with retractive radical, it is lexicographic. On the other hand, $A/M$ is, up to a unique isomorphism, an MV-subalgebra of the standard MV-algebra $[0,1]$. Therefore, according to Theorem \ref{reprthlex}, $A/O_M\cong\Gamma(R_M\lex G(M/O_M),(1,0))$.

So, in accordance with the proof of Theorem \ref{reprthlex}, and with an abuse of notation, we can see each element $a/O_M$ of $A/O_M$ as
$$\frac{a}{O_M}=\left(\frac{a}{M},g_{aM}\right)=(\widehat{a}(M),g_{aM}).$$
\end{proof}

%We define the MV-sheaf $(F_A, \pi, \Max A)$ where
%$$F_A=\bigsqcup_{M\in \Max A} G(\Rad (A/O_M))= \bigsqcup_{M\in \Max A}G(M/O_M)$$

%We define, for every $a\in A$ the map
%$$\begin{array}
%{cccc}
%\widetilde{a}:&\Max A &\longrightarrow &F_A\\
 %&M&\longmapsto & (g_{aM}, M)
%\end{array}$$

%where $g_{aM}=(a/O_M \ominus a/M)-(a/M\ominus a/O_M).$\\
%\textbf{Idea 4:} \textbf{Sheaf a partir de Presheaf (funtorial)}

\section{MV-sheaf Representation}\label{reprsec}

In this section, we shall prove our main theorem. The most important tools for our representation, besides those presented in the previous sections, are some results on lexicographic MV-algebras \cite{dfl}, and the Filipoiu and Georgescu sheaf representation \cite{G-F}. Throughout this section, $A$ will always denote a locally retractive MV-algebra.

In order to represent this class of MV-algebras by means of an MV-space, let us consider the following functors.
\begin{enumerate}
  \item Let $(X, \tau)$ be an MV-topological space and $(X, \B(\tau))$ its corresponding skeleton topological space defined in \cite{rus2}, where $\B(\tau)=\tau\cap\{0,1\}^X$. As usual, we consider the posets $\tau$ and $\B(\tau)$ with their natural order as categories, that is, the objects are the elements of $\tau$ and $\B(\tau)$ respectively, and the morphisms are given by $\alpha\leq\beta$ in $\tau$ and $U\subseteq V$ in $\B(\tau)$, respectively. The following map obviously defines a covariant functor:
    $$\begin{array}{cccc}
\sk: & \tau & \lto & \B(\tau)\\
     & \alpha & \lmapsto & \supp(\alpha)
\end{array}$$
For $\alpha\leq\beta$, we have the unique morphism $\alpha\stackrel{f}{\lto}\beta$ in $\tau$, and its corresponding morphism $\supp(\alpha)\stackrel{\sk(f)}{\lto}\supp(\beta)$ in $\B(\tau)$ is also uniquely determined, because $\alpha\leq\beta$ implies $\supp(\alpha)\subseteq \supp(\beta)$.
\item According Filipoiu and Georgescu's representation \cite{G-F}, each MV-algebra $A$ can be represented as the MV-algebra of global sections of a sheaf whose stalks are local MV-algebras and the base space is the space of maximal ideals of $A$ with the Zariski topology, $\mathcal{O}(\Max A)$. The associated sheaf in that representation is the following contravariant functor:
    $$\begin{array}{cccc}
  \mathfrak{F}: & \mathcal{O}(\Max A)& \lto & \MV \\
& U & \lmapsto & A/O_U  \end{array},$$
where $O_U=\bigcap_{M\in U}O_M$, and the unique morphism between two open sets (if it exists) is sent to the natural projection between the corresponding quotient algebras.
\item We recall the category $\alg$ whose objects are Abelian $\ell$-groups and whose morphisms are $\ell$-group homomorphisms. The following mapping defines a functor from the category of MV-algebras to the category $\alg$:
  $$\begin{array}{cccc}
  \mathfrak{G}: & \MV& \lto & \alg\\
	& A & \lmapsto & G(\Rad A)
  \end{array},$$
where $G(\Rad A)$ is the Abelian $\ell$-group generated by the ordered cancellative monoid $(\Rad(A),\oplus,\0)$. Actually, $G(\Rad A)=\mathcal{D}(A)$ where $\mathcal{D}$ is the inverse of the functor $\Delta:\alg\lto\MV^{\text{perf}}$ between Abelian $\ell$-groups and perfect MV-algebras presented in \cite{dinletperf} (note that the group $\mathcal{D}(A)$ can be constructed for any MV-algebra $A$, not necessarily perfect). The action on morphisms of the functor $\mathfrak{G}$ is exactly the same as for $\mathcal{D}$.

   %Actually, $G(\Rad A)=\mathcal{D}(A)$ where $\mathcal{D}$ is the functor described in Section \ref{D} (note that the group $\mathcal{D}(A)$ can be constructed for any MV-algebra $A$, not necessarily perfect), and the action of the functor on morphisms is exactly the same as for $\mathcal{D}$. %Thus, in this section we use the notation $G(\Rad A)$ for clarity.
  \end{enumerate}

Now, for each $\alpha\in \tau_A$, by Proposition \ref{r(a)=supp(widehat(a)} we have that $\supp(\alpha)\in\mathcal{O}(\Max A)$. So, set $A_\alpha:=\mathfrak{F}(\supp(\alpha))$ for each $\alpha\in \tau_A$.

\begin{proposition}\label{sheafprop}
The mapping
$$\begin{array}{cccc}
\mathfrak{H}:& \tau_A & \lto & \alg\\
 & \alpha & \lmapsto & G(\Rad (A_\alpha))
 \end{array}$$
is an MV-sheaf.
\end{proposition}
\begin{proof}
$\mathfrak{H}$ is obviously an MV-presheaf. On the other hand, in the construction performed by Filipoiu and Georgescu, the stalks are the local algebras $A/O_M$. Then, for each $M\in \Max A$,
$$\lim_{M\in\supp(\alpha)}\mathfrak{F}(\supp(\alpha))=A/O_M$$
Such a limit can be extended to the presheaf $\mathfrak{H}$ on the category $\alg$, thus obtaining the following two limits:
$$\lim_{M\in\supp(\alpha)}\Rad(\mathfrak{F}(\supp(\alpha)))=\Rad(A/O_M)$$ and $$\lim_{M\in\supp(\alpha)}\mathfrak{H}(\alpha)=G(\Rad(A/O_M)).$$

Since $\Rad(A/O_M)=M/O_M$ for each $M\in\Max A$, it follows that $\mathfrak{H}$ is an MV-sheaf on $\alg$ where the stalks are the $\ell$-groups $G(M/O_M)$.
\end{proof}
%Now, let $A$ be an MV-algebra. Let $M$ be a maximal ideal of $A$. By Proposition \ref{O_P is primary}, we have that $A/O_M$ is a local MV-algebra. Note that $A/M \cong (A/O_M)/(M/O_M)$ and $(M/O_M, \oplus, 0)$ is a lattice ordered cancellative monoid. We can construct, in the usual manner, the lattice ordered group $G(M/O_M)$ generated by $(M/O_M, \oplus, 0)$.

%Let us suppose that $A/O_M$ has retractive radical for every $M\in \Max A$. Then, for each $M\in \Max A$:
 %\begin{enumerate}[(i)]
  %    \item $A/O_M$ is a lexicographic MV-algebra (because the class of local MV-algebras with retractive radical is strictly included in the class of lexicographic MV-algebras \cite{dfl}).
   %   \item $A/M \cong (A/O_M)/(M/O_M)$.
   %\item $(M/O_M, \oplus, \0)$ is a lattice ordered cancellative monoid.
   %(\textbf{for me: Di Nola's Notes pag.56. corollary 2.5.12})
 %\end{enumerate}
%  We construct, in the usual manner, the lattice ordered group $G(M/O_M)$ generated by $(M/O_M, \oplus, \0)$.

We shall now present the MV-sheaf space associated to the MV-sheaf above.

%an MV-sheaf space $(H_A,\pi,\Max A)$ whose total MV-space, $H_A$, will be the disjoint union of stalks $\mathfrak{H}_M=G(M/O_M)$  and $\pi$ will be the trivial projection.

\begin{proposition}\label{shsp}
Let
$$H_A=\{(g_{aM},M):a\in A, M\in\Max A\},$$
and
$$\begin{array}{cccc}
\pi:& H_A & \lto & \Max A\\
 & (g_{aM},M) & \lmapsto & M
 \end{array}.$$
Then the triple $(H_A,\pi,\Max A)$ is an MV-sheaf space whose total MV-space, $H_A$, is the disjoint union of the stalks $\mathfrak{H}_M=G(M/O_M)$.
\end{proposition}
\begin{proof}
For each $a\in A$ we define:
$$\begin{array}
{cccc}
\widetilde{a}:&\Max A &\longrightarrow &H_A\\
 &M&\longmapsto & (g_{aM}, M)
\end{array}$$
It is clear that $(\pi\circ\widetilde{a})(M)=\pi(g_{aM},M)=M$ for all $M\in\Max A.$

As usual in sheaf representations, we shall use $\{\widetilde{a}^{\to}(\widehat{b})\}_{a,b\in A}$ as a subbase for an MV-topology on $H_A$, where
$$\widetilde{a}^{\to}(\widehat{b})(g_{cM},M)=\bigvee_{\widetilde{a}(N)=(g_{cM},M)}\widehat{b}(N)=\left\{\begin{array}{ll} \widehat{b}(M) & \textrm{if } g_{aM}=g_{cM}\\ 0 & \textrm{otherwise}\end{array}\right..$$
Let us see that $\alpha_{a,b}:=\{M\in\Max A:g_{aM}=g_{bM}\}$ is an element of $\mathcal{O}(\Max A)$. If $g_{aM}=g_{bM}$, we have the following cases:
 \begin{itemize}
     \item[-] If $\frac{a}{O_M}=\frac{b}{O_M}$ then $\left(\frac{a}{M},g_{aM}\right)= \left(\frac{b}{M},g_{bM}\right)$. Hence $\frac{a}{M}= \frac{b}{M}$ and, therefore, $$\alpha_{a,b}= H(d(a,b))\in \mathcal{O}(\Max A).$$
     \item[-] If $\frac{a}{O_M}\neq\frac{b}{O_M}$ then necessarily $\frac{a}{M}\neq\frac{b}{M}$. That is $\frac{a}{M} < \frac{b}{M}$ or $\frac{b}{M} > \frac{a}{M}$. Since $A/O_M$ has a lexicographic order, this implies that $\frac{a}{O_M} < \frac{b}{O_M}$ or $\frac{b}{O_M} > \frac{a}{O_M}$. Therefore there exists $c\in A$ such that $\frac{a}{O_M}=\frac{b\oplus c}{O_M}$ or $\frac{a\oplus c}{O_M}=\frac{b}{O_M}$. Hence
         $$\begin{array}{l}\alpha_{a,b}=\bigcup_{c\in A}\{M\in\Max A:\frac{a}{O_M}=\frac{b\oplus c}{O_M}\} \cup \bigcup_{c\in A}\{M\in\Max A:\frac{a\oplus c}{O_M}=\frac{b}{O_M}\}=\\
         =\bigcup_{c\in A}H(d(a,b\oplus c))\cup \bigcup_{c\in A}H(d(a\oplus c,b))\in\mathcal{O}(\Max A)\end{array}.$$
   \end{itemize}
 %(esto me dira que lan MV-topologia definida es una S-topologia y esto garantizara la continuidad local de pi. Hacerlo!)
As a consequence, each $\alpha_{a,b}$ is an element of $\tau_A$, and this guarantees that $(H_A,\pi,\Max A)$ is indeed an MV-sheaf space.
\end{proof}

The MV-sheaf defined above is an MV-sheaf of lattice-ordered Abelian groups. We want to obtain a representation of the MV-algebra $A$ through this MV-sheaf.

First, let us consider for each $a\in A$, the function $\widetilde{a}$ restricting the codomain $H_A$ to its image $\Im(\widetilde{a})=\{(g_{aM}, M): M\in\Max A\}$. Actually, the new $\widetilde{a}$ acts exactly like the previous one on the elements of the domain, so we shall use the same notation for them. Then, we have the bijective maps:
$$\begin{array}
{cccc}
\widetilde{a}:&\Max A &\longrightarrow & \Im(\widetilde{a})\\
 &M&\longmapsto & (g_{aM}, M)
\end{array}$$
and for each basic open set $\widehat{a}$ in $\Max A$ we have the open fuzzy set $\widetilde{a}^\to(\widehat{a})$ in $H_A$ satisfying
$$\widetilde{a}^\to(\widehat{a})(g_{aM},M)= \frac{a}{M}, \text{ for each } (g_{aM},M)\in \Im(\widetilde{a}).$$
%For each $a\in A$, the function $\widetilde{a}:\Max A \lto \Im(\widetilde{a})$ is continuous. In fact, let $\beta\in \tau_{\Im(\widetilde{a})}$ then $\beta\circ\widetilde{a}=\beta\restr\Im(\widetilde{a})=\beta\wedge\supp(\widetilde{a}^\to(\widehat{a})\in \tau_{H_A}$.

Now, let us consider the inverse of the graphic of $\widetilde{a}^\to(\widehat{a})$ given by
$$\mathfrak{a}:=G^{-1}(\widetilde{a}^\to(\widehat{a}))=\left\{\left(\frac{a}{M}, g_{aM}\right)\right\}_{M\in\Max A}.$$

\begin{definition}\label{MV-alg A representada}
  Let $\mathfrak{A}=\{\mathfrak{a}:a\in A\}$ . We define the structure $(\mathfrak{A},\oplus,^*,\mathfrak{o})$ with the operations and the constant defined as follow:
	
for each $\mathfrak{a}, \mathfrak{b}\in \mathfrak{A},$
\begin{enumerate}[(i)]
  \item $\mathfrak{o}:= G^{-1}(\widetilde{\0}^\to(\widehat{\0}))$
  \item $\mathfrak{a}\oplus \mathfrak{b}:=G^{-1}(\widetilde{a\oplus b}^\to(\widehat{a\oplus b}))$
  \item $\mathfrak{a}^* := G^{-1}(\widetilde{a^*}^\to(\widehat{a^*}))$.
\end{enumerate}
\end{definition}

\begin{theorem}
  $(\mathfrak{A},\oplus,^*,\mathfrak{o})$ is an MV-algebra.
\end{theorem}
\begin{proof}
Let us see that $\mathfrak{A}$ satisfies the equations defining MV-algebras, as listed in \cite[Definition 1.1.1]{mvbook}.
%First, we show that $(\mathfrak{A},\oplus,\mathfrak{o})$ is a commutative monoid.
  \begin{enumerate}[(MV1)]
          \item \begin{eqnarray*}(\mathfrak{a}\oplus \mathfrak{b})\oplus \mathfrak{c}= \left\lbrace\left(\frac{(a\oplus b)\oplus c}{M},g_{((a\oplus b)\oplus c) M}\right)\right\rbrace_{M\in\Max A}\\=\left\lbrace\left(\frac{a\oplus (b\oplus c)}{M},g_{(a\oplus (b\oplus c)) M}\right)\right\rbrace_{M\in\Max A}=
  \mathfrak{a}\oplus (\mathfrak{b}\oplus \mathfrak{c})\end{eqnarray*}
          \item $$\mathfrak{a}\oplus \mathfrak{b}= \left\lbrace\left(\frac{a\oplus b}{M},g_{(a\oplus b) M}\right)\right\rbrace_{M\in\Max A}=\\=\left\lbrace\left(\frac{b\oplus a}{M},g_{(b\oplus a) M}\right)\right\rbrace_{M\in\Max A}= \mathfrak{b}\oplus \mathfrak{a}$$
          \item \begin{eqnarray*}\mathfrak{a}\oplus \mathfrak{o}= \left\lbrace\left(\frac{a\oplus \0}{M},g_{(a\oplus \0) M}\right)\right\rbrace_{M\in\Max A}=\left\lbrace\left(\frac{a}{M},g_{a M}\right)\right\rbrace_{M\in\Max A}= \mathfrak{a}\end{eqnarray*}
              \item \begin{eqnarray*}(\mathfrak{a^*})^*=\left(\left\lbrace\left(\frac{a^*}{M},g_{(a^*) M}\right)\right\rbrace_{M\in\Max A}\right)^*=\left\lbrace\left(\frac{(a^*)^*}{M},g_{(a^*)^* M}\right)\right\rbrace_{M\in\Max A}\\= \left\lbrace\left(\frac{a}{M},g_{a M}\right)\right\rbrace_{M\in\Max A}= \mathfrak{a}\end{eqnarray*}
  \item \begin{equation*}
  \begin{split}\mathfrak{a}\oplus \mathfrak{o}^* & = \left\lbrace\left(\frac{a}{M},g_{a M}\right)\right\rbrace_{M\in\Max A}\oplus \left\lbrace\left(\frac{\0^*}{M},g_{0^* M}\right)\right\rbrace_{M\in\Max A}\\ & =\left\lbrace\left(\frac{a\oplus \0^*}{M},g_{(a\oplus \0^*) M}\right)\right\rbrace_{M\in\Max A}=\left\lbrace\left(\frac{\0^*}{M},g_{\0^* M}\right)\right\rbrace_{M\in\Max A}\\ & = \mathfrak{o}^*.
  \end{split}
  \end{equation*}
  \item \begin{equation*}
  \begin{split}
  (\mathfrak{a}^*\oplus \mathfrak{b})^*\oplus \mathfrak{b} &=\\&= \left\lbrace\left(\frac{(a^*\oplus b)^*}{M},g_{(a^* \oplus b)^*M}\right)\right\rbrace_{M\in\Max A}\oplus \left\lbrace\left(\frac{b}{M},g_{b M}\right)\right\rbrace_{M\in\Max A}\\&=\left\lbrace\left(\frac{(a^*\oplus b)^*\oplus b}{M},g_{((a^* \oplus b)^*\oplus b)M}\right)\right\rbrace_{M\in\Max A}\\&= \left\lbrace\left(\frac{(b^*\oplus a)^*\oplus a}{M},g_{((b^* \oplus a)^*\oplus a)M}\right)\right\rbrace_{M\in\Max A}\\&=\left\lbrace\left(\frac{(b^*\oplus a)^*}{M},g_{(b^* \oplus a)^*M}\right)\right\rbrace_{M\in\Max A}\oplus \left\lbrace\left(\frac{a}{M},g_{a M}\right)\right\rbrace_{M\in\Max A}\\&=(\mathfrak{b}^*\oplus \mathfrak{a})^*\oplus \mathfrak{a}.
  \end{split}
  \end{equation*}
\end{enumerate}
\end{proof}

\begin{theorem}\label{repr}
The MV-algebras $A$ and $\mathfrak{A}$ are isomorphic. So, any locally retractive MV-algebra is isomorphic to the algebra of global sections of an MV-sheaf of Abelian $\ell$-groups.
\end{theorem}
\begin{proof}
The natural map $\Psi:a \in A \lmapsto \mathfrak{a} \in \mathfrak{A}$ preserves the operations $\oplus, ^*$, and the constant $\0$. Indeed, by Definition \ref{MV-alg A representada}, for each $a,b\in A$, $\Psi(a\oplus b)=G^{-1}(\widetilde{a\oplus b}^\to(\widehat{a\oplus b}))=\mathfrak{a}\oplus \mathfrak{b}=\Psi(a)\oplus \Psi(b)$, and $\Psi(a^*)= \mathfrak{a}^*=(\Psi(a))^*$. Analogously, we have that $\Psi(\0)=\mathfrak{o}$. It is clear that $\Psi$ is a surjection, and let us prove that $\Psi$ is injective. Let $a,b\in A$, and suppose that $\Psi(a)=\Psi(b)$, that is,
$$\left\{\left(\frac{a}{M}, g_{aM}\right)\right\}_{M\in\Max A}= \left\{\left(\frac{b}{M}, g_{bM}\right)\right\}_{M\in\Max A},$$ then for each $M\in\Max A$,

$$\left(\frac{a}{M}, g_{aM}\right)=\left(\frac{b}{M}, g_{bM}\right),$$ that is, $\frac{a}{O_M}=\frac{b}{O_M}$ for every $M\in\Max A$, then $d(a,b)\in O_M$ for every $M\in\Max A$. So $a=b$ because $\bigcap\{O_M:M\in\Max A\}=0$.
\end{proof}

%\section{Conclusion}
%\label{concsec}

Theorem \ref{repr} gives us a complete representation of locally retractive MV-algebras. On the other hand, thanks to \cite[Theorem 4.5]{dinespger} and Lemma \ref{prodmvl}, we obtain the following two immediate corollaries.

\begin{corollary}\label{retremb}
Every MV-algebra $A$ is isomorphic to a subalgebra of an algebra with retractive radical.
\end{corollary}
\begin{proof}
Since $\bigcap_{M \in \Max A} O_M = \{0\}$, $A$ is subdirectly embeddable in $\prod_{M \in \Max A} A/O_M$, and $A/O_M$ is a local MV-algebra, for all $M \in \Max A$, because the $O_M$'s are primary ideals. By Theorem 4.5 of \cite{dinespger}, each $A/O_M$ can be embedded in a local MV-algebra with retractive radical $B_M$, hence $\prod_{M \in \Max A} A/O_M$ is embeddable in the algebra $ B = \prod_{M \in \Max A} B_M$ which, on its turn, has retractive radical by Lemma \ref{prodmvl}. The assertion follows.
\end{proof}

It is worth noticing that Corollary \ref{retremb} could be obtained also as an immediate consequence of Di Nola's Representation Theorem \cite{din}, since powers of ultrapowers of $[0,1]$ have retractive radical. Nonetheless, the way such an embedding is obtained in our proof, makes the embedding of the following result much clearer; moreover, the embeddability of any MV-algebra in an algebra with retractive radical has never been explicitly stated before, to the best of our knowledge.

\begin{corollary}\label{sheafemb}
Every MV-algebra $A$ can be embedded in the algebra of global sections of an MV-sheaf of Abelian $\ell$-groups.
\end{corollary}
\begin{proof}
With reference to the previous proof, $A$ embeds in $B$ which is isomorphic to the algebra of global section $\mathfrak{B}$ as in Definition \ref{MV-alg A representada}.
\end{proof}

\section*{Compliance with Ethical Standards}
\begin{itemize}
\item Funding: This work was supported by Colciencias Ph.D. scholarship doctoral scholarship “Doctorado Nacional-567” (Luz Victoria De La Pava) and FAPESB Grant No. APP0072/2016 (Ciro Russo).
\item Conflict of Interest: Luz Victoria De La Pava declares that she has no conflict of interest. Ciro Russo declares that he has no conflict of interest.
\item Ethical approval: This article does not contain any studies with human participants or animals performed by any of the authors.
\end{itemize}

\end{document}